\documentclass[12pt,reqno]{amsart}
\usepackage{amsmath}
\usepackage{amssymb}
\usepackage[left=3cm,top=3cm,right=3cm,bottom=3cm]{geometry}
\usepackage{epsfig}

\begin{document}
\newtheorem{theorem}{Theorem}[section]
\newtheorem{lemma}[theorem]{Lemma}
\newtheorem{definition}[theorem]{Definition}
\newtheorem{conjecture}[theorem]{Conjecture}
\newtheorem{proposition}[theorem]{Proposition}
\newtheorem{algorithm}[theorem]{Algorithm}
\newtheorem{corollary}[theorem]{Corollary}
\newtheorem{observation}[theorem]{Observation}
\newtheorem{problem}[theorem]{Open Problem}
\newtheorem{remark}[theorem]{Remark}
\newcommand{\noin}{\noindent}
\newcommand{\ind}{\indent}
\newcommand{\om}{\omega}
\newcommand{\I}{\mathcal I}
\newcommand{\pp}{\mathcal P}
\newcommand{\ppp}{\mathfrak P}
\newcommand{\N}{{\mathbb N}}
\newcommand{\LL}{\mathbb{L}}
\newcommand{\R}{{\mathbb R}}
\newcommand{\E}{\mathbb E}
\newcommand{\Prob}{\mathbb{P}}
\newcommand{\eps}{\varepsilon}

\title{Metric dimension for random graphs}

\author{B\'ela Bollob\'as}
\address{Department of Pure Mathematics and Mathematical Statistics,
Wilberforce Road, Cambridge, CB3 0WA, UK, and Department of
Mathematical Sciences, University of Memphis, Memphis, TN 38152, USA,
and London Institute for Mathematical Sciences, 35a South Street,
London, W1K 2XF, UK.}
\email{\tt B.Bollobas@dpmms.cam.ac.uk}

\author{Dieter Mitsche}
\address{Universit\'{e} de Nice Sophia-Antipolis, Laboratoire J-A Dieudonn\'{e}, Parc Valrose, 06108 Nice cedex 02}
\email{\texttt{dmitsche@unice.fr}}

\author{Pawe\l{} Pra\l{}at}
\address{Department of Mathematics, Ryerson University, Toronto, ON, Canada}
\email{\tt pralat@ryerson.ca}

\thanks{The first author was partially supported by NSF grant DMS-0906634 and EU MULTIPLEX grant 317532. The second and third author gratefully acknowledge support from NSERC, MPrime, and Ryerson University.}

\maketitle

\begin{abstract}
The metric dimension of a graph $G$ is the minimum number of vertices in a subset $S$ of the vertex set of $G$ such that all other vertices are uniquely determined by their distances to the vertices in $S$. In this paper we investigate the metric dimension of the random graph $G(n,p)$ for a wide range of probabilities $p=p(n)$.
\end{abstract}

\section{Introduction}\label{sec:intro}
Let $G=(V,E)$ be a finite, simple, connected graph with $|V|=n$ vertices.
For a subset $R \subseteq V$ with $|R|=r$, and a vertex $v \in V$, define $d_R(v)$ to be the $r$-dimensional vector whose $i$-th coordinate $(d_R(v))_i$ is the length of the shortest path between $v$ and the $i$-th vertex of $R$. We call a set $R \subseteq V$  a \emph{resolving set} if for any pair of vertices $v, w \in V$, $d_R(v) \neq d_R(w)$. Clearly, the entire vertex set $V$ is always a resolving set, and so is $R=V\setminus \{z\}$ for every vertex $z$.   The \emph{metric dimension} $\beta(G)$ (or simply $\beta$, if the graph we consider is clear from the context) is then the smallest cardinality of a resolving set. We have the trivial inequalities $1 \leq \beta(G) \leq n-1$, with the lower bound  attained for a path, and the upper bound for the complete graph.

The problem of studying the metric dimension was proposed in the mid-1970s by Slater~\cite{Sla75}, and Harary and Melter~\cite{Har76}. As a start, Slater~\cite{Sla75} determined the metric dimension of trees. Two decades later, Khuller,  Raghavachari and Rosenfeld~\cite{Kul96} gave a linear-time algorithm for computing the metric dimension of a tree, and characterized the graphs with metric dimensions $1$ and $2$. Later on, Chartrand,  Eroh, Johnson and Oellermann~\cite{CEJO00} gave necessary and sufficient conditions for a graph $G$ to satisfy $\beta(G)=n-1$ or $\beta(G)=n-2$.

Denoting by $D=D(G)$ the diameter of a graph $G$, it was observed in~\cite{Kul96} that $n \leq D^{\beta-1}+\beta$. Recently, Hernando,  Mora,  Pelayo,  Seara and Wood~\cite{HMPSW10} proved that $n \leq (\lfloor \frac{2D}{3}\rfloor +1)^{\beta}+\beta \sum_{i=1}^{\lceil D/3 \rceil} (2i-1)^{\beta-1}$, and  gave extremal constructions that show that this bound was sharp. Moreover, in~\cite{HMPSW10} graphs of metric dimension $\beta$ and diameter $D$ were characterized.

The metric dimension of the cartesian product of graphs was investigated by C\'{a}ceres, Hernando et al.~\cite{Cac07}, and the relationship between $\beta(G)$ and the {\em determination number} of $G$ (the smallest size of a set $S$ such that every automorphism of $G$ is uniquely determined by its action on $S$) was studied by C\'{a}ceres, Garijo et al.~\cite{CGPS10}. Also, Bailey and Cameron~\cite{Bailey11} studied the metric dimension of groups, and the relationship of the problem of determining $\beta (G)$ to the graph isomorphism problem.

Concerning algorithms, the problem of finding the metric dimension is known to be NP-complete for general graphs (see~\cite{GJ79, Kul96}). Recently, D\'{\i}az et al.~\cite{Diaz12} showed that determining $\beta(G)$ is NP-complete for planar graphs, and gave a polynomial-time algorithm for determining the metric dimension of an outerplanar graph. Furthermore, in~\cite{Kul96} a polynomial-time algorithm approximating $\beta(G)$ within a factor  $2 \log n$ was given. On the other hand, Beerliova et al.~\cite{Beerliova} showed that the problem is inapproximable within $o(\log n)$ unless P=NP. Hauptmann et al.~\cite{Hauptmann} then strengthened the result and showed that unless NP $\subseteq$ DTIME$(n^{\log \log n})$, for any $\eps > 0$, there is no $(1-\eps) \log n$-approximation for determining $\beta(G)$, and finally Hartung et al.~\cite{Hartung12} extended the result by proving that the metric dimension problem is still inapproximable within a factor of $o(\log n)$ on graphs with maximum degree three.

\bigskip

In this paper, we consider the metric dimension of the classical binomial random graph $G(n,p)$. As usual 
(see, for example,~\cite{bol, JLR}), the space ${\mathcal G}(n,p)$ of random graphs is the probability triple $(\Omega, \mathcal{F}, \Prob)$ where $\Omega$ is the set of all graphs with vertex set $[n]=\{1,2,\dots,n\}$, \ $\mathcal{F}$ is the family of all subsets of $\Omega$, and  ${\Prob}$ is the probability measure on $(\Omega, \mathcal{F})$ defined by
$$
\Prob(G) = p^{|E(G)|} (1-p)^{{n \choose 2} - |E(G)|} \,.
$$
A random graph $G(n,p)$ is simply a random point of this space. Clearly, $G(n,p)$ can be obtained by
${n \choose 2}$ independent coin flips, one for each unordered pair of vertices, with probability of `success' $p$: if the flip corresponding to a pair $(x,y)$ is `success', then we join $x$ to $y$, otherwise we do not join them. We shall take $p=p(n)$ to be a function of $n$; in particular, $p$ may tend to zero as $n$ tends to infinity. All asymptotics throughout are as $n \rightarrow \infty $. We say that an assertion concerning $G(n,p)$ holds \emph{asymptotically almost surely} (\emph{a.a.s.}) if the probability that it holds tends to $1$ as $n$ goes to infinity.

\bigskip

As far as we know, not much is known about the metric dimension of $G(n,p)$. Babai et al.~\cite{Babai80} showed that in $G(n,1/2)$ a.a.s.\ the set of $\lceil (3\log n)/\log 2 \rceil$ vertices with the highest degrees can be used to test whether two random graphs are isomorphic (in fact, they provided an $O(n^2)$ algorithm to do it), and hence they obtained an upper bound of $\lceil (3\log n)/\log 2 \rceil$ for the metric dimension of $G(n,1/2)$ that holds a.a.s. Frieze et al.~\cite{Frieze07} studied sets resembling resolving sets, namely \emph{identifying codes}: a set $C \subseteq V$ is an identifying code of $G$, if $C$ is a dominating set (every vertex $v \in V \setminus C$ has at least one neighbour in $C$) and $C$ is also a separating set (for all pairs $u,v \in V$, one must have $N[u] \cap C \neq N[v] \cap C$, where $N[u]$ denotes the closed neighbourhood of $u$). Observe that a graph might not have an identifying code, but note also that for random graphs with diameter $2$ the concepts are very similar. The existence of identifying codes and bounds on their sizes in $G(n,p)$ were established in~\cite{Frieze07}.  The same problem in the model of random geometric graphs was analyzed by M\"uller and Sereni~\cite{Mueller09}, and Foucaud and Perarnau~\cite{Foucaud12} studied the same problem in random $d$-regular graphs.

\bigskip

Let us collect our results into a single theorem covering all random graphs with expected average degree $d=pn(1+o(1)) \gg \log^5 n$ and expected average degree in the complement of the graph $(n-1-d) = (1-p)n(1+o(1)) \ge (3 n \log \log n)/\log n$. In later sections we shall prove slightly stronger results for specific ranges of $p$. For a visualization of the behaviour of $\log_n \beta(G(n,n^{x-1}))$ see also Figure~\ref{fig1}(a) and the description right after the statement of the theorem. 

The intuition behind the theorem is the following: if a random graph is sufficiently dense, then the graph locally (that is, ``observed'' from a given vertex) ``looks'' the same. In other words, the cardinality of the set of vertices at a certain graph distance from a given vertex $v$ does not differ much for various $v$. After grouping the vertices according to their graph distances from $v$, it turns out that for the metric dimension the ratio between the sizes of the two largest groups of vertices is of crucial importance. If these two groups are roughly of the same size, then a typical vertex added to the resolving set distinguishes a lot of pairs of vertices, and hence the metric dimension is small. If, on the other hand, these two groups are very different in size, a typical vertex distinguishes those few vertices belonging to the second largest group from the rest. The number of other pairs that are distinguished is negligible and hence the metric dimension is large. 

It is clear that this parameter is non-monotonic. Let us start with a random graph with constant edge probability $p$. For each vertex $v$ in the graph, a constant fraction of all vertices are neighbours of $v$ and a constant fraction of vertices are non-neighbours. When decreasing $p$, the number of neighbours decreases, and some vertices will appear at graph distance $3$. As a result, the metric dimension increases. Continuing this process, the number of vertices at graph distance $3$ increases more and more, and at some point this number is comparable to the number of vertices at graph distance $2$. Then, the metric dimension is small again, and the same phenomenon appears in the next iterations.

\bigskip

The precise statement is the following.

\begin{theorem}\label{thm:main}
Suppose that
$$
\log^5 n \ll d=p(n-1) \le n \left( 1 - \frac{3 \log \log n}{\log n} \right).
$$
Let $i \ge 0$ be the largest integer such that $d^i = o(n)$, let $c = c(n) = d^{i+1}/n$, and let
$$
q=
\begin{cases}
(e^{-c})^2+(1-e^{-c})^2 & \text{ if } p = o(1) \\
p^2 + (1-p)^2 & \text{ if } p = \Theta(1).
\end{cases}
$$
For $i \ge 1$, let $\eta = \log_n d^i $. Finally, let $G = (V,E) \in G(n,p)$. Then, the following assertions hold a.a.s.
\begin{itemize}
\item [(i)] If $c = \Theta(1)$, then
$$
\beta(G) = (1+o(1)) \frac {2 \log n}{\log (1/q)} = \Theta(\log n).
$$
\item [(ii)] If $c \to \infty$ and $e^c \le (\log n)/(3 \log \log n)$, then
$$
\beta(G) =(1+o(1)) \frac{2\log n}{\log (1/q)}=(1+o(1))e^c \log n \gg \log n.
$$
\item[(iii)] If $e^c > (\log n)/(3 \log \log n)$, then
$$
\left( \eta +o(1) \right) \left( \frac {d^i}{n} + e^{-c} \right)^{-1} (\log n) \le \beta(G)  \le (1+o(1)) \left( \frac {d^i}{n} + e^{-c} \right)^{-1} (\log n).
$$
In particular,
\[
\beta(G) =
\begin{cases}
\Theta(e^c \log n)=\Theta(\frac{\log n}{\log (1/q)})  & \text{ if } e^{-c} = \Omega( d^i / n ) \\
\Theta( \frac {n \log n}{d^i})  & \text{ if } e^{-c} \ll d^i / n ,
\end{cases}
\]
and hence in all cases we have $\beta(G) \gg \log n.$
\end{itemize}
\end{theorem}

\begin{remark} 
Note that here and in the following results, it follows from the definition of $i$ that $c=d^{i+1}/n=\Omega(1)$. Furthermore, 
$$
\eta =\log_n d^i \ge \frac {i}{i+1}+o(1) \ge \frac 12+o(1), 
$$
where the first inequality follows from the fact that 
$$
\log_n d^i = \frac {\log d^i} {\log n} = \frac {\log d^i}{\log (d^{i+1}/c)} = \frac {i \log d}{ (i+1) \log d - \log c} \ge \frac {i}{i+1} + o(1).
$$
\end{remark}

Observe that Theorem~\ref{thm:main} shows that $\beta(G)$ undergoes a ``zigzag'' behaviour as a function of $p$. It follows that a.a.s.\ $\log_n \beta(G(n,n^{x-1}))$ is asymptotic to the function $f(x)=1-x \lfloor 1/x \rfloor$ shown in Figure~\ref{fig1}(a). Indeed, for cases (i) and (ii) we have $c=n^{o(1)}$ (that is, $d=n^{(1+o(1))/i}$ for some $i \in \N$) and a.a.s.\ $\beta(G) = n^{o(1)}$.  This corresponds to a collection of points $(1/i, 0)$, $i \in \N$ in the figure. For ranges of $p$ considered in case (iii), we have that a.a.s.\ $\beta(G)$ is of order $\left( {d^i}/{n} + e^{-c} \right)^{-1}$. For $d = n^{x+o(1)}$, where $1/(i+1) < x < 1/i$ for some $i \in \N$, it follows that a.a.s.\ $\beta(G) = \Theta( n/d^i )= n^{1-ix+o(1)}$, which corresponds to linear parts of the function $f(x)$ of slope $1-ix$. The function $f(x)$ is hence not continuous at $x=i$, $i \in \N \setminus \{1\}$.

The result is asymptotically tight for sparse graphs (that is, for $d = n^{o(1)}$). The ratio between our upper and lower bound is at most $2+o(1)$ and follows another ``zigzag'' function $f(x)=(x \lfloor 1/x \rfloor)^{-1}$ shown in Figure~\ref{fig1}(b). Indeed, for cases (i) and (ii) we obtained an asymptotic behaviour of $\beta(G)$. This corresponds to a collection of points $(1/i, 0)$, $i \in \N$ in the figure. In case (iii) the ratio is asymptotic to $\eta^{-1}$. For $d = n^{x+o(1)}$, where $1/(i+1) < x < 1/i$ for some $i \in \N$, $\eta^{-1} = \eta^{-1}(x) \sim 1/(ix) \le (i+1)/i$. Hence, $\eta^{-1} \sim (x \lfloor 1/x \rfloor)^{-1}$.

\begin{figure}[h]
\begin{center}
\begin{tabular}{cc}
\includegraphics[width=2.5in,height=2in]{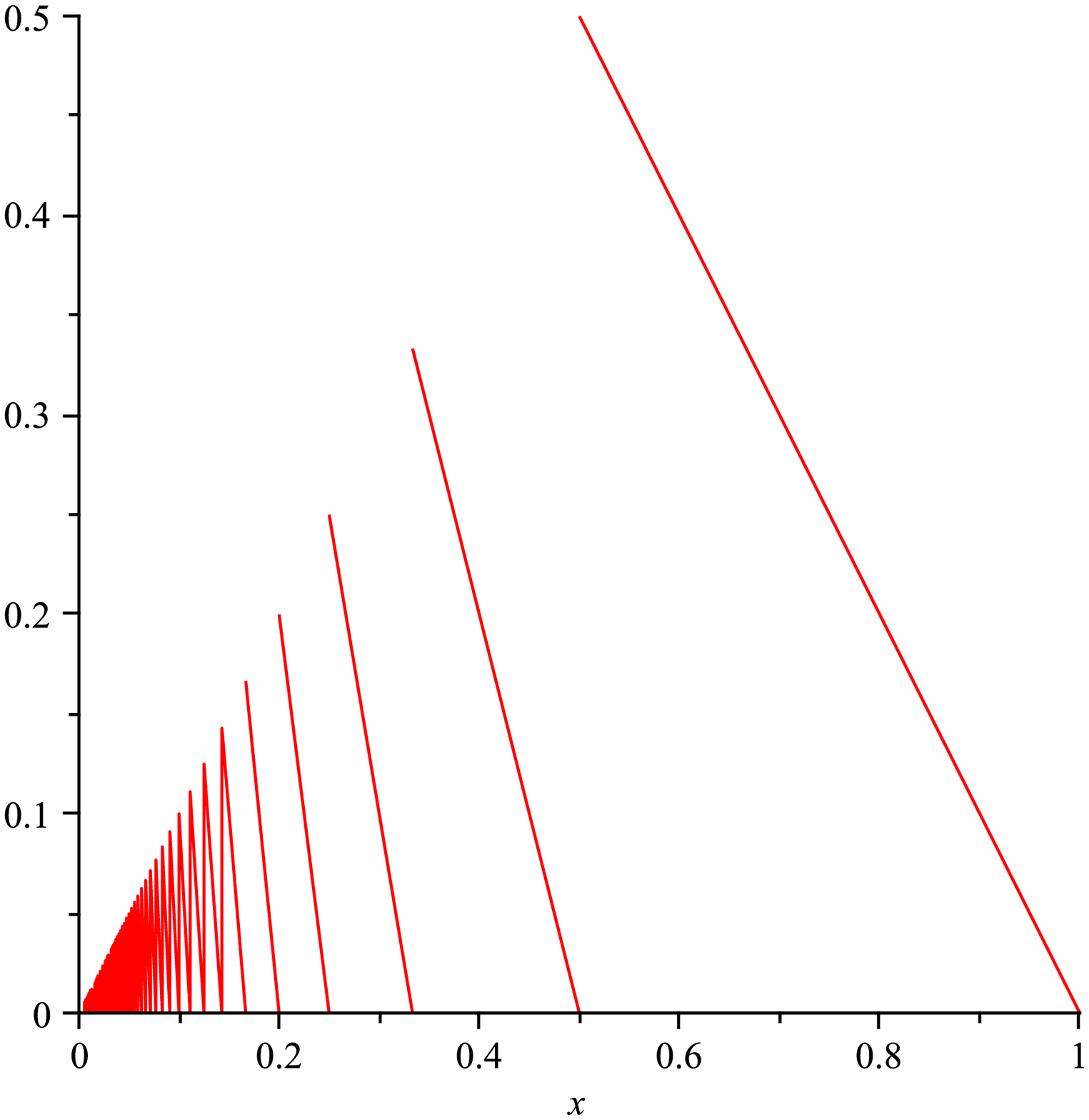} &
\includegraphics[width=2.5in,height=2in]{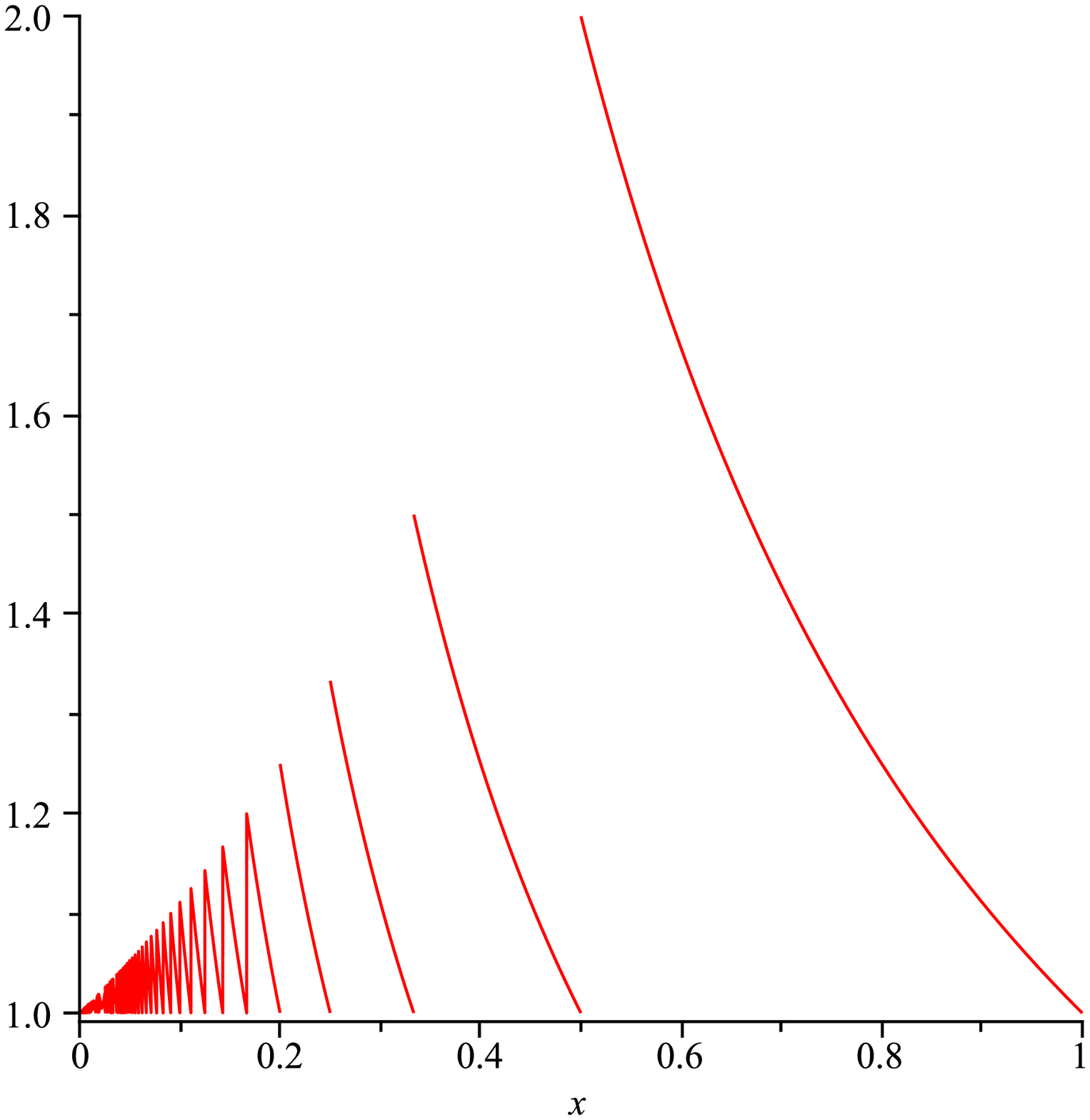} \\
(a) the `zigzag' function  & (b) the upper/lower bound ratio \\
$f(x) = 1-x \lfloor 1/x \rfloor$ & $\eta^{-1}(x) = (x \lfloor 1/x \rfloor)^{-1}$ \\
\end{tabular}
\end{center}
\caption{}\label{fig1}
\end{figure}

\section{Expansion properties}

Let us start with the following expansion-type properties of random graphs. For a vertex $v \in V$, let $S(v,i)$ and $N(v,i)$ denote the set of vertices at distance $i$ from $v$ and the set of vertices at distance at most $i$ from $v$, respectively. For any $V' \subseteq V$, let $S(V',i) = \bigcup_{v \in V'} S(v,i)$ and $N(V',i) = \bigcup_{v \in V'} N(v,i)$.

\begin{lemma}\label{lem:gnp exp}
Let $\omega=\omega(n)$ be a function tending to infinity with $n$ such that $\omega \le (\log n)^4 (\log \log n)^2$.  Then the following properties hold a.a.s. for $G(n,p)$.
\begin{itemize}
\item [(i)] Suppose that $\omega \log n \le d=p(n-1) = o(n)$. Let $V' \subseteq V$ with $|V'| \le 2$ and let $i \in \N$ such that $d^i = o(n)$. Then,
$$
\left| S(V',i) \right| = \left(1+O \left( \frac {1}{\sqrt{\omega}}\right) + O\left( \frac{d^i}{n} \right) \right) d^i |V'|.
$$
In particular, for every $x,y \in V$ ($x \neq y$) we have
$$
\left| S(x,i) \setminus S(y,i) \right| = \left(1+O \left( \frac {1}{\sqrt{\omega}}\right) + O\left( \frac{d^i}{n} \right) \right) d^i.
$$
\item [(ii)] Suppose that $(\omega \log^2 n)/(\log \log n) \le d=p(n-1) = o((n \log \log n) / (\log n)^2)$. Let $R \subseteq V$ with $r=|R| \le (\log n)^2 / (\log \log n)$, $x \in V \setminus R$, and let $i \in \N$ such that $r d^i = o(n)$. Then,
$$
\left| S(x,i) \setminus N(R,i) \right| = \left(1+O \left( \sqrt{ \frac {\log \log n}{\omega \log n}}\right) + O\left( \frac {1}{\omega} \right) + O\left( \frac{r d^i}{n} \right) \right) d^i.
$$
\end{itemize}
\end{lemma}

\begin{proof}
For (i), we will show that a.a.s.\ for every $V' \subseteq V$ with $|V'| \le 2$ and $i \in \N$ we have the desired concentration for $|S(V',i)|$, provided that $d^i = o(n)$. The statement for any pair of vertices $x,y$ will follow immediately (deterministically) from this.

In order to investigate the expansion property of neighbourhoods, let $Z \subseteq V$, $z=|Z|$, and consider the random variable $X = X(Z) = |N(Z,1)|$. We will bound $X$ in a stochastic sense. There are two things that need to be estimated: the expected value of $X$, and the concentration of $X$ around its expectation.

Since for $x=o(1)$ we have $(1-x)^z = e^{-xz(1+O(x))}$ and also $e^{-x}=1-x+O(x^2)$, it is clear that
\begin{eqnarray}
\E [X] &=& n - \left(1- \frac {d}{n-1} \right)^z (n-z) \nonumber \\
&=& n - \exp \left( - \frac {dz}{n} (1+O(d/n)) \right) (n-z) \nonumber \\
&=& dz (1+O(dz/n)), \label{eq:expX}
\end{eqnarray}
provided $dz = o(n)$. We next use a consequence of Chernoff's bound (see e.g.~\cite[p.\ 27, Corollary~2.3]{JLR}), that
\begin{equation}\label{chern}
\Prob( |X-\E [X]| \ge \eps \E [X]) ) \le 2\exp \left( - \frac {\eps^2 \E [X]}{3} \right)
\end{equation}
for  $0 < \eps < 3/2$.

This implies that, for $\eps = 2/{\sqrt{\omega}}$, the expected number of sets $V'$ satisfying  
$$
\big| |N(V',1)| - \E[|N(V',1)|] \big| > \eps d|V'|
$$ 
and $|V'| \le 2$ is  at most 
\[
\sum_{z \in \{1,2\}} 2 n^z \exp \left( - \frac {\eps^2 z d }{3+o(1)} \right) \le \sum_{z \in \{1,2\}} 2 n^z \exp \left( - \frac {\eps^2 z \omega \log n}{3+o(1)} \right) = o(1),
\]
since $d \geq \omega \log n$. Hence the statement holds for $i=1$ a.a.s. Now, we will estimate the cardinalities of $N(V',i)$ up to the $i$'th iterated neighbourhood, provided $d^i = o(n)$ and thus $i = O(\log n /\log \log n)$. It follows from~(\ref{eq:expX}) and~(\ref{chern}) (with $\eps = 4 (\omega |Z|)^{-1/2}$) that in the case $\omega \log n / 2 \le |Z| = o(n/d)$ with probability at least $1-n^{-3}$
$$
|N(Z,1)| = d |Z| \left( 1+ O \left( d|Z|/n \right) + O\left((\omega |Z|)^{-1/2} \right) \right),
$$
where the bounds in $O()$ are uniform. As we want a result that holds a.a.s., we may assume this statement holds deterministically, since there are only $O(n^2 \log n)$ choices for $V'$ and $i$. Given this assumption, we have good bounds on the ratios of the cardinalities of $N(V',1)$, $N(N(V',1),1) = N(V',2)$, and so on. Since $i=O(\log n / \log \log n)$ and $\sqrt{\omega} \le (\log n)^2 (\log \log n)$, the cumulative multiplicative error term is
\begin{align*}
(1+&O(d/n) + O(1/\sqrt{\omega})) \prod_{j=2}^i \left( 1+ O \left( d^j/n \right) + O\left( \omega^{-1/2} d^{-(j-1)/2} \right) \right) \\
&= (1+O(1/\sqrt{\omega}) + O(d^i/n) )  \prod_{j=7}^{i-3} \left( 1+ O \left( \log^{-3} n \right) \right) = (1+O(1/\sqrt{\omega}) + O(d^i/n) ),
\end{align*}
and the proof of part (i) is complete.

\bigskip

Now, let us move to part (ii). Exactly the same strategy as for part (i) is used here and so we only outline the proof by pointing out the differences. Every time the $i$'th neighbourhood of $x$ is about to be estimated, we first expose the neighbourhood of $R$. Since for each vertex $v$, a.a.s. $|N(v)|=d(1+o(1))$, during this process, only $|N(R,i)|=O(rd^i)$ vertices are discovered. Now, the neighbourhood $N(x,i-1)$ is expanded, but this time the vertices of $N(R,i)$ need to be excluded from the consideration. However, the expected size of $S(x,i) \setminus N(R,i)$ is affected by a factor of $(1+O(rd^i/n))$ only, and so this causes no problem (since the very same error term comes from~(\ref{eq:expX})). As before, the largest error term for the expectation appears for the largest possible value of $i$. The concentration follows from~(\ref{chern}), and thus (again, exactly as before) the error term in the concentration result is the largest for $i=1$. Using that $r \leq (\log^2 n)/(\log \log n)$, the expected number of pairs $(x,R)$ for which the statement fails for $i=1$ is, by applying~(\ref{chern}) with $\eps = 2 /\sqrt{\omega'}$, at most
\begin{equation}\label{densecase}           
n^{\frac{\log^2 n}{\log \log n}+1} \exp \left( - \frac {\eps^2 d}{3+o(1)} \right) \le \exp \left( \frac{(1+o(1)\log^3 n}{\log \log n} - \frac {\eps^2 d}{3+o(1)} \right) = o(1),
\end{equation}
provided $d \ge (\omega' \log^3 n)/ (\log \log n)$ for some function $\omega'$ tending to infinity with $n$. (Note that the condition for $d$ here is slightly stronger than the one we want to have.) Hence, by Markov's inequality, a.a.s.\ we are guaranteed to have an error term of $(1+O (1/\sqrt{\omega'}) + O(r d^i/n) )$ for any pair $(x,R)$. Therefore, it follows from the fact that $\omega \leq (\log n)^4 (\log \log n)^2$ that the error term $O(\sqrt{\log \log n/(\omega \log n)})$ is of order at least  $1/\sqrt{\omega'}$ if $\omega' \geq (\log n)^{10}(\log \log n)^2$. The proof of part (ii) is finished for $d \ge (\log^{13} n) (\log \log n)$.

Next, we shall concentrate on $(\omega \log^2 n)/(\log \log n) \le d = n^{o(1)}$ and shall obtain a slightly better error term for the case $i=1$. It is well known (and can be easily shown using Markov's inequality) that a.a.s.\ there is no $K_{2,3}$ in $G$. Conditioning on this we get from part (i) that
$$
\left| S(x,1) \setminus N(R,1) \right| = \left(1+O \left( \sqrt{ \frac {\log \log n}{\omega \log n}}\right) + O\left( \frac {|R|}{d} \right) \right) d,
$$
and the result holds for $i=1$, since $|R|/d \le \omega^{-1}$. As before, the error term is maximal for  $i=1$: indeed, for $i \ge 2$ we already showed that the expectation of $|S(x,i) \setminus N(R,i)|$ can be estimated in the same way as the expectation of $|S(x,i)|$, since its size is not affected by disregarding $N(R,i)$. In particular, for $i=2$ we have the expected size of $S(x,2) \setminus N(R,2)$  to be equal to
$$
\left(1+O \left(\frac{rd^2}{n} \right) \right) d \left| S(x,1) \setminus N(R,1) \right| =(1+o(1))d^2.
$$
The expected number of pairs $(x,R)$ for which the statement fails for $i=2$ is, by applying~(\ref{chern}) with $\eps = 1 /\omega$, at most
\begin{eqnarray}       
n^{\frac{\log^2 n}{\log \log n}+1} \exp \left( - \frac {\eps^2 d^2}{3+o(1)} \right) &\le& \exp \left( \frac{(1+o(1))\log^3 n}{\log \log n} - \frac {\log^4 n}{(3+o(1))(\log \log n)^2} \right) \nonumber \\
&=& o(1), \label{densecase2}    
\end{eqnarray}
and by Markov's inequality we are guaranteed an error term of $(1+O (1/\omega) + O(r d^2/n))$ for any pair $(x,R)$ and $i=2$. 
By the same argument as in part (i), the cumulative multiplicative error term is a product of the error term for $i=1$ (from the concentration) and the last $i$ (from the expectation), and the proof of part (ii) is complete.
\end{proof}

\section{Upper bound}\label{sec:upper_bound}

In this section, we shall prove an upper bound for the metric dimension, coming from an application of the probabilistic method.

\begin{theorem}\label{thm:upper bound}
Suppose that $d=p(n-1) \gg \log^3 n$ and $n - d \gg \log n$. Let $i \ge 0$ be the largest integer such that $d^i = o(n)$, let $c = c(n) = d^{i+1}/n$, and let
$$
q=
\begin{cases}
(e^{-c})^2+(1-e^{-c})^2 & \text{ if } p = o(1) \\
p^2 + (1-p)^2 & \text{ if } p = \Theta(1).
\end{cases}
$$
Finally, let $G = (V,E) \in G(n,p)$. Then, the following assertions hold a.a.s.
\begin{itemize}
\item [(i)] If $c = \Theta(1)$, then
$$
\beta(G) \le (1+o(1)) \frac {2 \log n}{\log (1/q)} = \Theta(\log n).
$$
\item [(ii)] If $c \to \infty$, then
$$
\beta(G) \le (1+o(1)) \left( \frac {d^i}{n} + e^{-c} \right)^{-1} (\log n) \gg \log n.
$$
In particular,
$$
\beta(G) \le
\begin{cases}
(1+o(1)) e^c \log n =(1+o(1)) \frac { 2 \log n}{\log (1/q)} & \text{ if } e^{-c} \gg d^i / n \\
(1+o(1)) \frac {n \log n}{d^i} & \text{ if } e^{-c} \ll d^i / n.
\end{cases}
$$
\end{itemize}
\end{theorem}

\begin{proof}
Let us start with the following useful observation.

\textbf{Claim}: Suppose that a (deterministic) graph $G=(V,E)$ on $n$ vertices satisfies the following property: for a given pair of vertices $x,y \in V$, the probability that a random set $W$ of cardinality $w$ does \emph{not} distinguish $x$ and $y$ is at most $1/n^2$. (For different pairs of vertices, the set $W=W(x,y)$ is generated independently.) Then, the metric dimension is at most $w$.

\emph{Proof of the Claim}: The claim clearly holds by the probabilistic argument. Indeed, since the expected number of pairs that are not distinguished by a random set $W$ is at most $1/2$, there is at least one set $W$ that distinguishes \emph{all} pairs. 

\smallskip

Now, we are going to show that a.a.s.\ a random graph satisfies some expansion property, and then we will show that any (deterministic) graph $G$ with this property must also satisfy the assumption of the claim (for some $w$ to be determined soon), and so must have $\beta(G) \le w$. The conclusion will be then that a.a.s.\ $\beta(G) \le w$ for $G \in G(n,p)$.

Let $\eps > 0$ be any constant (at the end, we will let $\eps \to 0$ slowly), and fix a pair of vertices $x,y \in V$ of $G = (V,E) \in G(n,p)$. Suppose first that $i=0$; that is, $p = \Theta(1)$. Note that any vertex that is adjacent to $x$ but not to $y$ (or vice versa) distinguishes this pair. We expect $2p(1-p)(n-2) \ge (2+o(1)) (\omega \log n)$ of such vertices and so a.a.s.\ for every pair of two vertices we have
$$
X := \left| \Big(S(x,1) \setminus N(y,1) \Big) \cup \Big(S(y,1) \setminus N(x,1) \Big) \right| = (2+o(1)) p(1-p)n
$$
by~(\ref{chern}) (applied with $\eps=3/\sqrt{\omega}$). 

Finally, consider any deterministic graph for which this property holds for all pairs $x$ and $y$. Let $p_w$ be the probability that a random set $W$ of cardinality $w$ does not distinguish the pair under consideration. We get that
$$
p_w \le \left( \frac{n-X}{n} \right) \left( \frac {n-1-X}{n-1} \right) \cdots \left( \frac {n-w+1-X}{n-w+1} \right) \le \left(1 - \frac {X}{n} \right)^{w} = q^{w(1+o(1))},
$$
which is at most $1/n^2$ for $w = (2+\eps) (\log n) / (\log (1/q))$. The claim implies that $\beta(G) \le w$, and the result follows for $i=0$.

\smallskip

Suppose now that $i \ge 1$; that is, $p=o(1)$. As before, we are going to use the claim to show the desired bound. However, this time the expansion property will be different. Clearly, if there exists $z \in V$ such that $z \in S(x,j) \setminus S(y,j)$ for some $j \in \N$, then $z$ distinguishes the pair. It follows from Lemma~\ref{lem:gnp exp}(i) that a.a.s.\ for every pair $x,y$ we have $|S(x,i) \setminus S(y,i)| = (1+o(1)) d^i$, and so, by symmetry, there are $(2+o(1)) d^i$ vertices in the $i$'th neighbourhood of $x$ or $y$ that can distinguish this pair (this is the first type of vertices which is able to distinguish $x$ and $y$). 

Now, let us focus on distinguishing vertices of the second type. Any vertex $z$ at distance $i+1$ from $x$ but at distance at least $i+2$ from $y$, or vice versa, also distinguishes $x$ and $y$. Suppose first that $c \le 0.51 \log n$ which in turn implies that $d^i/n = c/d = o(\log^{-2} n)$. By Lemma~\ref{lem:gnp exp}(i), since $d \gg \log^3 n$ and hence $\omega \gg \log^2 n$, and using as before that for $p=o(1)$, $1-p=e^{-p+O(p^2)}$, we expect
\begin{align*}
2(1-p)&^{(1+o(\log^{-1} n)+O(d^i/n) )d^i} \left( 1 - (1-p)^{(1+o(\log^{-1} n)+O(d^i/n) )d^i} \right) n (1+o(1)) \\
&= 2 \exp \left( -(1+o(\log^{-1} n)) c \right) \left( 1 - \exp \left( -(1+o(\log^{-1} n)) c \right) \right)n(1+o(1)) \\
&= 2 e^{-c+o(1)} (1-e^{-c+o(1)}) n \\
&= (2+o(1)) e^{-c} (1-e^{-c}) n
\end{align*}
vertices of this type. (This is the place where we need to control error terms by concentrating on graphs that are dense enough.) Since the expectation is $\Omega(n^{0.49})$, it follows from Chernoff's bound~(\ref{chern}) that with probability $1-o(n^{-2})$ the cardinality is well concentrated around its expectation. On the other hand, if $c > 0.51 \log n$, then the contribution from this second group is bounded by $(2+o(1))e^{-c}n \le 3n^{0.49}$. This can be ignored since the contribution from the first group is at least $(2+o(1))d^i = \Omega(\sqrt{n})$.  We get that with probability $1-o(n^{-2})$ the number of vertices that can distinguish the pair $x$ and $y$ is at least $(2+o(1)) \left( d^i + e^{-c} (1-e^{-c}) n \right)$. Hence, a random graph a.a.s.\ has this expansion property for all pairs $x,y$.

Now, as before, we consider any deterministic graph with the mentioned expansion property, and show that for a given pair $x,y$, the probability $p_w$ that a random set of cardinality $w$ ($w$ will be determined soon) does not distinguish this pair is at most $1/n^2$. For $c=\Theta(1)$ we get that
\[
p_w \le \left(1-(2+o(1)) \left( \frac {d^i}{n} + e^{-c} (1-e^{-c}) \right) \right)^w = \left(1 - 2 e^{-c} (1-e^{-c}) \right)^{w(1+o(1))} = q^{w(1+o(1))},
\]
which is at most $1/n^2$ for $w = (2+\eps) \log n / \log (1/q)$. If $c \to \infty$, then
\[
p_w \le \exp \left(-(2+o(1)) \left( \frac {d^i}{n} + e^{-c} \right) w \right) \le n^{-2}
\]
for $w = (1+\eps) ( d^i/n + e^{-c} )^{-1} \log n$. The desired bound is implied by the claim. 

As we promised, we let $\eps$ to tend to zero (slowly) and the proof is complete.
\end{proof}

\section{Lower bounds}\label{sec:lower_bound}

In order to show lower bounds, we shall make use of the following well-known result proved in~\cite[Theorem~6]{bol_diameter} for graphs with average degree $d = p(n-1)$ that tends to infinity faster than $\log^3 n$. Moreover, in~\cite[Corollary~10.12]{bol} the condition was relaxed and it is now required only that $d \gg \log n$. Recall that $D=D(G)$ is the diameter of a graph $G$.


\begin{lemma}[\cite{bol}, Corollary 10.12]\label{lem:diameter}
Suppose that $d = p (n-1) \gg \log n$ and 
\[
d^i/n - 2 \log n \to \infty \text{ \ \ \ \ and \ \ \ \ } d^{i-1}/n - 2 \log n \to -\infty.
\]
Then the diameter of  $G(n,p)$ is equal to $i$ a.a.s.
\end{lemma}

\bigskip

Let $i \ge 0$ be the largest integer such that $d^i = o(n)$, and let $c = c(n) = d^{i+1}/n$. Now, we are ready to show that the upper bound for the metric dimension is asymptotically tight if $c \le \log \log n - \log \log \log n - \log 3$ (see Theorem~\ref{thm:lower_bound_Suen}); otherwise, there is at most a constant factor difference  (see Theorems~\ref{thm:lower_bound_diam} and~\ref{thm:lower_bound_diam2}).
\begin{theorem}\label{thm:lower_bound_Suen}
Let $\eps = \eps(n) = (3 \log \log n) / (\log n) = o(1)$. Suppose that $\log^5 n \ll d=p(n-1) \le n(1-\eps)$. Let $i \ge 0$ be the largest integer such that $d^i = o(n)$, let $c = c(n) = d^{i+1}/n$, and let
$$
q=
\begin{cases}
(e^{-c})^2+(1-e^{-c})^2 & \text{ if } p = o(1) \\
p^2 + (1-p)^2 & \text{ if } p = \Theta(1).
\end{cases}
$$
Suppose that $e^{c} \le \eps^{-1} = (\log n)/(3 \log \log n)$. Finally, let $G = (V,E) \in G(n,p)$. Then, a.a.s.
$$
\beta(G) \ge (1+o(1)) \frac {2 \log n}{\log (1/q)}.
$$
\end{theorem}

\begin{proof}
Our goal is to show that a.a.s.\ there is no resolving set $R$ of cardinality
$$
r := \frac{(2-\eps)\log n}{\log (1/q)} \le (1+o(1)) \frac{\log n}{\eps} \le \frac {\log^2 n}{\log \log n}.
$$
The probability that a given set $R$ of cardinality $r$ forms a resolving set has to be estimated from above. We will use Suen's inequality that was introduced in~\cite{Suen} and revised in~\cite{Janson_Suen}, and then the result will follow after applying the union bound. First we consider the case $p=o(1)$ (that is, $i \ge 1$); the differences in the case $p=\Theta(1)$ (that is, $i=0$) will be carefully discussed afterwards.

%

By Lemma~\ref{lem:gnp exp}(i) applied with $\omega \gg \log^4 n$, a.a.s.\ for every $v \in V$, we have 
\begin{eqnarray*}
|S(v,i)| &=& d^i (1+o(\log^{-2} n)). 
\end{eqnarray*}
 Hence,  by repeatedly applying Lemma~\ref{lem:gnp exp}(i) for all vertices of $R$, it follows that a.a.s. for all $R \subseteq V$ with $|R|=r$ we have 
$$|N(R,i)| = O(d^i r) = O(cnr/d) = O(n \log^2 n / d) = o(n).$$ 
Moreover, by Lemma~\ref{lem:gnp exp}(ii), this time applied with $\omega \gg (\log n)^3 (\log \log n)$,  it also follows that a.a.s. for all $R$ and all $v \in R$ we have
\begin{eqnarray*}
|S(v,i) \setminus N(R \setminus \{v\},i)| &=& d^i (1+o(\log^{-2} n)).
\end{eqnarray*}
Hence, a.a.s.\ there is no set $R$ without these expansion properties and so we may assume below that they all hold.
  
Fix any $R \subseteq V$ with $|R|=r$. Expose $N(R,i)$ and let $S = V \setminus N(R,i)$ be the set of vertices at distance at least $i+1$ from $R$ (note that no edge within $S$ and no edge between $S$ and $S(R,i)$ is exposed yet). From the previous observation we assume that  $|S|=(1+o(1))n$. 

Let $\I = \{ (x,y) : x, y \in S, \, x \neq y \}$, and for any $(x,y) \in \I$, let $A_{x,y}$ be the event (with the corresponding indicator random variable $I_{x,y}$) that $d_R(x)=d_R(y)$. Let $X = \sum_{(x,y) \in \I} I_{x,y}$. Clearly, the probability that $R$ is a resolving set is at most the probability that $X=0$. The associated \emph{dependency graph} has $\I$ as its vertex set, and $(x_1, y_1) \sim (x_2, y_2)$ if and only if $\{ x_1, y_1 \} \cap \{ x_2, y_2 \} \neq \emptyset$. It follows from Suen's inequality that
\begin{equation}\label{eq:Suen}
\Prob(X=0) \le \exp \left( - \mu + \Delta e^{2\delta} \right),
\end{equation}
where
\begin{eqnarray*}
\mu &=& \sum_{(x,y) \in \I} \Prob(A_{x,y})\\
\Delta &=& \sum_{ (x_1, y_1) \sim (x_2, y_2)} \E[I_{x_1,y_1} I_{x_2,y_2}]\\
\delta &=& \max_{(x_1,y_1) \in \I} \sum_{(x_2, y_2) \sim (x_1, y_1) } \Prob(A_{x_2, y_2}).
\end{eqnarray*}

We will first estimate $\mu$.  
For a given vector $d \in \{i+1,i+2\}^r$, let $R_{i+1} = R_{i+1}(d) \subseteq R$ be a set of vertices of $R$ that we want to be at distance exactly $i+1$ from $x$ and $y$ (recall that $x,y \notin N(R,i)$); $R_{i+2} = R_{i+2}(d) := R \setminus R_{i+1}$. Since we want to have a lower bound on the probability that $x$ and $y$ yield the same vector $d$, in order to avoid additional complications,  for every vertex $v \in R_{i+1}$ we can ignore possible edges between  $x$, $y$ and a vertex in  $S(v,i) \cap S(R \setminus \{v\},i)$ and consider only the possible edges between $x$, $y$ and $S(v,i) \setminus S(R \setminus \{v\},i)$. Let $p_{x,y}(d)$ be the probability that the distance from both $x$ and $y$  to $R_{i+1}$ is $i+1$ and to $R_{i+2}$ is \emph{at least} $i+2$. We have that 
\begin{eqnarray*}
p_{x,y}(d) &\ge& \left( (1-p)^{|S(R_{i+2},i)|} \prod_{v \in R_{i+1}} \left(1-(1-p)^{|S(v,i) \setminus S(R \setminus \{v\},i)|} \right) \right)^2\\
&\ge& (1+o(1)) \left( \exp \Big(- c |R_{i+2}| (1+o(\log^{-2} n)) \Big) \left( 1 - \exp \Big(- c (1+o(\log^{-2} n)) \Big)^{|R_{i+1}|}  \right)  \right)^2 \\
&=& (1+o(1)) \left( e^{-c} \right)^{2|R_{i+2}|} \left( 1 - e^{-c} \right)^{2|R_{i+1}|},
\end{eqnarray*}
since $c |R| = O(\log^2 n)$. (This is the place where we need to control error terms by concentrating on graphs that are dense enough.) 
Finally, it is straightforward to see that a.a.s.\ both $x$ and $y$ are adjacent to at least one vertex of $S(v,i+1)$ for every $v \in R$. (Recall that it follows from Lemma~\ref{lem:diameter} that the diameter of $G$ is $i+2$ a.a.s., and so that there are only two possible distances that occur in $d_R(x)$ for any $x \in S$.) Hence, $\Prob(d_R(x)=d_R(y)=d) = (1+o(1)) p_{x,y}(d)$.

Since there are ${r \choose k}$ vectors with exactly $k$ entries equal to $(i+2)$,
\begin{eqnarray*}
\mu &=& {n-o(n) \choose 2} \sum_{k=0}^r {r \choose k} (1+o(1)) \left( e^{-c} \right)^{2 k} \left( 1 - e^{-c} \right)^{2(r-k)} \\
&=& (1+o(1)) {n \choose 2} \left( (e^{-c})^2 + (1-e^{-c})^2 \right)^r = (1+o(1)) {n \choose 2} q^r.
\end{eqnarray*}
By the same calculations we have
\begin{eqnarray*}
\Delta &=& (1+o(1)) 3 {n \choose 3} \left( (e^{-c})^3 + (1-e^{-c})^3 \right)^r\\
\delta &=& (1+o(1)) 2 n \left( (e^{-c})^2 + (1-e^{-c})^2 \right)^r = (2+o(1)) n q^r.
\end{eqnarray*}

Now we are ready to apply Suen's inequality~(\ref{eq:Suen}). Since $q^r = n^{-2+\eps}$ and using that $1-x \le e^{-x}$, we get
\begin{eqnarray*}
\log \left( \Prob(X=0) \right) &\le& - (1+o(1)) {n \choose 2}  q^r \left( 1 - n \left( \frac {(e^{-c})^3 + (1-e^{-c})^3}{(e^{-c})^2 + (1-e^{-c})^2} \right)^r e^{O(nq^r)} \right) \\
&=& - (1+o(1)) \frac{n^{\eps}}{2} \left( 1 - n \left( 1 - \frac {(e^{-c}) - (e^{-c})^2}{(e^{-c})^2 + (1-e^{-c})^2} \right)^r e^{O(n^{\eps-1})} \right) \\
&\le& - \frac{n^{\eps}}{3} \left( 1 - n \left( 1 - \frac {1-q}{2q} \right)^r \right)\\
&\le& - \frac{n^{\eps}}{3} \left( 1 - n \exp \left( - \frac {1-q}{2q} \cdot \frac{(2-\eps)\log n}{\log (1/q)} \right) \right).
\end{eqnarray*}
Note that the function $f(q):=(q-1)/(q \log q)$ is decreasing in $(0,1)$ and tends to 1 as $q \to 1$. Since $1-q=2e^{-c}(1-e^{-c}) \ge (1+o(1)) 2 \eps$, or equivalently, $q \le 1-(1+o(1))2\eps$, the minimum is attained at  $1-q_0$ such that $q_0 = (2+o(1)) \eps$. Therefore, using the fact that for $-1 \le x < 1$, $\log(1-x)=x+x^2/2+O(x^3)$, we have
$$
\frac {1-q}{2q} \cdot \frac{2-\eps}{\log (1/q)} \ge \frac {q_0}{2(1-q_0)} \cdot \frac {2-\eps}{q_0 (1+q_0/2+O(q_0^2))} = 1+ \left( \frac {1}{2} + o(1) \right) \eps,
$$
which in turn implies that, say, $\Prob( X=0 ) \le \exp (- n^{\eps}/4)$. Finally, the expected number of resolving sets $R$ of size $r$ is at most
$$
{n \choose r} \exp \left( -\frac {n^{\eps}}{4} \right) \le \exp \left( O \left( \frac {\log^3 n}{\log \log n} \right) - \Omega \left( \log^3 n \right)\right) = o(1),
$$
and the result follows by Markov's inequality.

Now we point out the differences with the (slightly easier) case $i = 0$. In this case we have
\begin{eqnarray*}
\mu &=&  (1+o(1)) {n \choose 2} \left( p^2 + (1-p)^2 \right)^r = (1+o(1)) {n \choose 2} q^r.
\end{eqnarray*}
and 
\begin{eqnarray*}
\Delta &=& (1+o(1)) 3 {n \choose 3} \left( p^3 + (1-p)^3 \right)^r\\
\delta &=& (1+o(1)) 2 n \left( p^2 + (1-p)^2 \right)^r = (2+o(1)) n q^r.
\end{eqnarray*}
Then, again using $q^r = n^{-2+\eps}$,
\begin{eqnarray*}
\log \left( \Prob(X=0) \right) &\le& - (1+o(1)) {n \choose 2}  q^r \left( 1 - n \left( \frac {p^3 + (1-p)^3}{p^2 + (1-p)^2} \right)^r e^{O(nq^r)} \right) \\
&=& - (1+o(1)) \frac{n^{\eps}}{2} \left( 1 - n \left( 1 - \frac {p - p^2}{p^2 + (1-p)^2} \right)^r e^{O(n^{\eps-1})} \right) \\
&\le& - \frac{n^{\eps}}{3} \left( 1 - n \left( 1 - \frac {1-q}{2q} \right)^r \right)\\
&\le& - \frac{n^{\eps}}{3} \left( 1 - n \exp \left( - \frac {1-q}{2q} \cdot \frac{(2-\eps)\log n}{\log (1/q)} \right) \right).
\end{eqnarray*}
Now, by assumption, $p \le (1-\eps)$, and therefore $1-q=2p(1-p) \ge (1+o(1)) 2 \eps$. Hence, the minimum of $f(q)=(q-1)/(q \log q)$ is attained at  $1-q_0$ such that $q_0 = (2+o(1)) \eps$, and the remaining calculations can be performed as before.
\end{proof}

\bigskip

The next two theorems show that in all other cases we consider, the ratio between the upper and the lower bounds is at most $(2+o(1))$. We will assume until the end of this section that $d=o(n)$, as for $d=\Omega(n)$ Theorem~\ref{thm:lower_bound_Suen} can be applied.

\begin{theorem}\label{thm:lower_bound_diam}
Suppose that $\log n \ll d=p(n-1) = o(n)$.  Let $i \ge 1$ be the largest integer such that $d^i = o(n)$, let $c = c(n) = d^{i+1}/n$, and $\eta = \log_n d^i$. Suppose that $c - 2 \log n \to \infty$. Finally, let $G = (V,E) \in G(n,p)$. Then, a.a.s.
$$
\beta(G) \ge (\eta+o(1)) \frac {n \log n}{d^i}.
$$
\end{theorem}

\begin{proof}
Put $\omega = \omega(n) := d/(\log n) \to \infty$, and let $\eps = \eps(n) > 0$ be any function tending (slowly) to zero such that $\omega^{-1/2} = o(\eps)$, $d/n = o(\eps)$, and $\eps \gg (\log \log n)/(\log n)$. We will show that a.a.s.\ no $R$ of cardinality $r = (\eta-\eps) (n \log n)/d^i$ is a resolving set.

The general approach is similar to the previous proof. As before, we are going to test all sets $R \subseteq V$ with $|R|=r$ and partition them into 4 bins.\\
-- \emph{Bin 1} contains sets with $|N(R,i-1)| > (\eta-\eps/2) \frac {n \log n}{d}$;\\
-- \emph{Bin 2} contains sets with $|N(R,i-1)| \le (\eta-\eps/2) \frac {n \log n}{d}$ and there is at most one vertex at distance at least $i+1$ from $R$;\\
-- \emph{Bin 3} contains sets with $|N(R,i-1)| \le (\eta-\eps/2) \frac {n \log n}{d}$, there are at least two vertices at distance at least $i+1$ from $R$, and there is at least one vertex at distance at least $i+2$ from $R$;\\
-- \emph{Bin 4} contains all remaining sets.

The observation is that when all sets under consideration are in Bin 4, then there is no resolving set of cardinality $r$. Indeed, if this property holds, then for every set $R$ with $|R|=r$ there are at least two vertices at distance exactly $i+1$ from \emph{every} vertex from $R$ and hence they cannot be distinguished by $R$. 

It follows from Lemma~\ref{lem:gnp exp}(i) (after applying it $r$ times, for each vertex of $R$) that a.a.s.\ for all $R \subseteq V$ with $|R|=r$ we have
$$
|N(R,i-1)| \le r d^{i-1} (1+O(\omega^{-1/2})) \le (\eta-\eps/2) \frac {n \log n}{d} = O \left( \frac {n}{\omega} \right).
$$
Hence, a.a.s.\ Bin 1 is empty. Similarly, it follows from Lemma~\ref{lem:diameter} that the diameter of $G$ is $i+1$ a.a.s. so Bin 3 is empty a.a.s. It remains to show that a.a.s.\ Bin 2 is empty.

Fix any $R \subseteq V$ with $|R|=r$ and perform BFS until $N(R,i-1)$ is discovered. Since here we count sets with good expansion properties, we may assume that $|N(R,i-1)| \le (\eta-\eps/2) \frac {n \log n}{d}$. Now, we are going to estimate the probability that there are at least two vertices in $V \setminus N(R,i-1)$ that are not adjacent to any vertex in $N(R,i-1)$. 
Noticing that $d/n=o(\eps)$ and using $1-x \le e^{-x}$, we get that the probability that at most one vertex is not adjacent to $N(R,i-1)$ (we have at most $n$ choices for this vertex) is at most 
\begin{align*}
n &\left( 1 - (1-p)^{|N(R,i-1)|} \right)^{n(1-O(1/\omega))} \\
&\le n \left( 1 - (1-p)^{(\eta-\eps/2) \frac {n \log n}{d}} \right)^{n(1-O(1/\omega))} \\
&= n \left( 1 - \exp \Big(-(\eta-\eps/2+o(\eps)) \log n \Big) \right)^{n(1-O(1/\omega))} \\
&\le n \left( 1 - \exp \Big(-(\eta-\eps/3) \log n \Big) \right)^{n(1-O(1/\omega))} \\
&= \exp \left( \log n - n^{1-\eta+\eps/3}(1-o(1)) \right) \\
&= \exp \left( \log n - n^{1-\eta+\eps/4} n^{\eps/12} (1-o(1)) \right) \\
&\le \exp \left( - n^{1-\eta+\eps/4} \right).
\end{align*}
 (The last line follows since $\eps \gg (\log \log n)/(\log n)$ implies that 
$$
n^{\eps/12}(1-o(1)) > \exp \big( (\eps/12) \log n - 1 \big) > \exp(2 \log \log n) = \log^2 n,
$$
and so replacing $\eps/3$ by $\eps/4$ is enough to make the additive $\log n$-term to be negligible.)

Hence, the probability that $R$ belongs to Bin 2 is at most $\exp \left( - n^{1-\eta+\eps/4} \right)$. On the other hand, the number of possible choices for $R$ is equal to 
$$
{n \choose r} \le n^{(\eta-\eps) \frac {n \log n}{d^i}} \le \exp \left( \frac {n (\log n)^2}{d^i}\right) \le \exp \left( n^{1-\eta+(2\log \log n)/(\log n)} \right),
$$
and so Bin 2 is empty a.a.s.\ after applying the union bound. The result follows.
\end{proof}

\bigskip

The next theorem deals with slightly smaller values of $c$. Since the diameter might change from $i+1$ to $i+2$ in this situation, a more careful treatment is required.

\begin{theorem}\label{thm:lower_bound_diam2}
Suppose that $\log^3 n \ll d=p(n-1) = o(n)$.  Let $i \ge 1$ be the largest integer such that $d^i = o(n)$, and let $c = c(n) = d^{i+1}/n$. Suppose that $e^{c} > (\log n)/(3 \log \log n)$ (in particular, $c \to \infty$) and $c \le 3 \log n$. Finally, let $G = (V,E) \in G(n,p)$. Then, a.a.s.
$$
\beta(G) \ge \left( \frac{i}{i+1} + o(1) \right) \left( \frac {d^i}{n} + e^{-c} \right)^{-1} (\log n).
$$
\end{theorem}

\begin{proof}
The proof is similar to the one used to prove Theorem~\ref{thm:lower_bound_diam}. This time the diameter is a.a.s.\ at most $i+2$ by Lemma~\ref{lem:diameter} (in fact, it is a.a.s.\ equal to $i+2$, provided that $c-2\log n \to -\infty$) so the proof has to be slightly adjusted. However, we will show that a.a.s.\ for every set $R$ of the desired cardinality there are at least two vertices at distance $i+1$ from \emph{every} vertex of $R$ (it is clear that these vertices cannot be distinguished by $R$).

Let $\omega = \omega(n) = \min\{ (d/\log^3 n)^{1/2}, (\log n)^{1/2} \} \to \infty$. We will consider two cases independently. 

\emph{Case 1}: Suppose first that $e^{-c}n = A d^i$ for some $A=A(n)=\Omega(1)$ ($A$ might tend to infinity); in particular, $c \le \log n$. Fix $\eta > 0$ and take any $R \subseteq V$ with 
$$
|R|=r := \left( \eta - \omega^{-1/2} \right) e^c (\log n) = (1+o(1)) \eta e^c (\log n).
$$
As before, (based on Lemma~\ref{lem:gnp exp}(i) applied independently $r$ times, since $d \gg \log^3 n$ and $d^{i-1}/n = c / d^2 = o(\log^{-1} n)$), we may assume that 
$$
|N(R,i-1)| \le (1+o(\log^{-1} n)) rd^{i-1} = \left( \eta - (1+o(1)) \omega^{-1/2} \right) (n \log n)/(Ad) = o(n),
$$
since a.a.s.\ there is no $R$ that violates this condition. Hence, after exposing edges from $S(R,i-1)$ to $S(R,i)$, we get that the probability that a vertex outside of  $N(R,i-1)$ is not adjacent to any vertex in $S(R,i-1)$ is 
\begin{eqnarray*}
p_\ell &=& (1-p)^{|S(R,i-1)|} \\
&\ge& \exp \left( - \left( \eta - (1+o(1)) \omega^{-1/2} \right) A^{-1} (\log n)(1+O(d/n)) \right) \\
&=& \exp \left( - \left( \eta - (1+o(1)) \omega^{-1/2} \right) A^{-1} (\log n) \right) \\ 
&=& \exp \left( \left( - \eta/A + (1+o(1)) \omega^{-1/2}/A \right) (\log n)  \right),
\end{eqnarray*}
where the third line follows from the fact that $d/n = c/d^i \le c/d \ll \log^{-2}n$, and hence the term $\exp(O((\log n) d/n))=\exp(o(1)) = (1+o(1))$, and it is thus absorbed in the leading factor $(1+o(1)).$ 

Now, fix any $v \notin N(R,i)$. We will show that the probability of having all coordinates equal to $i+1$ is large enough and so with high probability there are at least two such vertices, which implies that $R$ is not a resolving set. For $x \in R$, let $A_x=A_x(v)$ be the event that $v \in S(x,i+1)$. As before, it follows from Lemma~\ref{lem:gnp exp}(i) that we may assume that $|S(x,i)|=(1+O( \omega^{-1} \log^{-1}n)) d^i$ (recall that $\omega \le (d/\log^3 n)^{1/2}$, and thus the first and more important error term of Lemma~\ref{lem:gnp exp}(i) is bounded by $O(1/\left( \omega (\log n)^{(3/2)} \right)$). Since $rd^{i} = (1+o(1)) \eta n \log n / A$ could be of order at least $n$, it may happen that $S(x,i)$ overlaps with neighbourhoods of other vertices of $R$. However, this actually helps, since the $A_x$'s ($x \in R$) are positively correlated (the fact that there is at least one edge from $v$ to $S(x_1,i)$ for some $x_1 \in R$ increases the chances that there is at least one edge from $v$ to $S(x_2,i)$ for some $x_2 \in R \setminus \{x_1\}$). Hence, the probability that $v$ has all coordinates equal to $i+1$ is
\begin{eqnarray*}
p_v = \Prob \left( \bigcap_{x \in R} A_x \right) &\ge& \left( 1- (1-p)^{(1+O(\omega^{-1} \log^{-1}n)) d^i} \right)^r \\
&=& \left( 1- e^{-c +O(\omega^{-1})} \right)^r \\
&=& \left( 1- (1+O(\omega^{-1})) e^{-c} \right)^r \\
&=&\exp \left( - (1+O(\omega^{-1})) e^{-c} r \right) \\
&=& \exp \left( - (\eta - (1+o(1)) \omega^{-1/2}) (\log n) \right),
\end{eqnarray*}
where the second line follows, since $e^{-c} \le (3 \log \log n)/(\log n)$,  and the second last equation follows since $O((e^{-c})^2 r)=O(e^{-c} r (\log \log n) / (\log n))=O(\omega^{-1}e^{-c}r)$, and thus the quadratic term $O(x^2 r)$ coming from the approximation $(1-x)^r=\exp(-xr+O(x^2 r))$ is already absorbed  in the given error term. 

Note that $p_v$ is independent of $p_{\ell}$, since for the bound on $p_v$ only previously unexposed edges between vertices from $S(R,i)$ and vertices not in $N(R,i)$ are taken into account. Hence, the expected number of vertices with all coordinates equal to $i+1$ is 
\begin{eqnarray*}
(1+o(1)) n p_\ell p_v &\ge& (1+o(1))\exp \Big( (1-\eta(1+1/A) + (1+o(1)) \omega^{-1/2} (1+1/A) ) (\log n) \Big) \\
&\ge& (1+o(1))n^{1-\eta \frac{A+1}{A}}\exp(\log^{3/4} n) \\
&\ge& n^{1- \eta \frac {A+1}{A}} (\log n)^4,
\end{eqnarray*}
where we used that $\omega \le (\log n)^{1/2}$ and $(1+o(1))\exp(\log^{3/4} n) \ge \log^4 n$. 
It follows from Chernoff's bound~(\ref{chern}) that with probability at most $\exp (- n^{1- \eta \frac {A+1}{A}} (\log n)^3)$ there are less than two such vertices, and hence this is also an upper bound for the probability that a given $R$ is a resolving set. Put $\eta = (\frac {i}{i+1}) (\frac {A}{A+1})$. We get that a.a.s.\ there is no resolving set of size $r$ by the union bound, since the number of sets of cardinality $r$ is at most
$$
n^r = \exp( r \log n) \le \exp( \eta e^c (\log n)^2) = \exp (O(n (\log n)^2 / d^i)) = \exp( O(n^{1-i/(i+1)} (\log n)^2) ),
$$
where the last equality follows since $d^i = \Omega(n^{\frac{i}{i+1}})$. 

\emph{Case 2}: Using exactly the same argument, one can deal with the case $e^{-c}n \ll d^i$ (but still it is assumed that $c \le 3 \log n$). We only point out the differences comparing to the previous case. This time we take
$$
r := \left( \eta - \omega^{-1/2} \right) \frac {n \log n}{d^i} = (1+o(1)) \eta \frac {n \log n}{d^i},
$$
and, using the same calculations as before, we may assume that
 $$|N(R,i-1)| \le \left(\eta - (1+o(1)) \omega^{-1/2} \right) (n \log n)/d,$$
and hence we obtain
$$
p_\ell \ge \exp \left( \left( - \eta + (1+o(1)) \omega^{-1/2} \right) (\log n) \right).
$$
Arguing as before, we obtain
$$
p_v \ge \left( 1- (1+O(\omega^{-1})) e^{-c} \right)^r = \exp \left( - O \left( \frac{e^{-c}n}{d^i} \right) (\log n) \right) \ge \exp ( - O(\omega^{-1} \log n) ),
$$
assuming additionally that $\omega^{-1} \ge \frac{e^{-c}n}{d^i}$, which we may, since $\frac{e^{-c}n}{d^i} = o(1)$; that is, $\omega = \min\{ (d/\log^3 n)^{1/2}, (\log n)^{1/2}, d^i e^c / n \} \to \infty$. As before, we conclude that 
$$
(1+o(1)) n p_\ell p_v \ge n^{1- \eta} (\log n)^4,
$$ 
and so the assertion holds for $\eta = \frac {i}{i+1}$, completing the proof of our theorem.
\end{proof}

\section{Acknowledgements} 

Some of the research was carried out while the first and the third authors were visiting the Theory Group at Microsoft Research, Redmond; we are grateful to Yuval Peres and other members of the group for their hospitality.

We also would like to thank the anonymous referee for suggestions which improved the exposition of our results.

\end{document}